\documentclass[12pt]{amsart}

\usepackage[utf8]{inputenc}
\usepackage{microtype}
\usepackage{amsfonts}
\usepackage{amsmath}
\usepackage{graphicx}
\usepackage{float}
\usepackage[dvipsnames]{xcolor}
\usepackage{hyperref}
\usepackage{tikz-cd}

\usepackage{amssymb}
\usepackage{enumitem}
\usepackage{mathrsfs}

\usepackage{tikz}
\usetikzlibrary{topaths}
\usetikzlibrary{calc}

\usepackage{a4wide}

\usepackage{empheq}

\newtheorem{theorem}{Theorem}[section]
\newtheorem*{theorem*}{Theorem}
\newtheorem{lemma}[theorem]{Lemma}
\newtheorem{proposition}[theorem]{Proposition}
\newtheorem*{proposition*}{Proposition}
\newtheorem{corollary}[theorem]{Corollary}
\newtheorem*{corollary*}{Corollary}

\newtheorem*{conjecture*}{Conjecture}
\newtheorem{question}[theorem]{Question}
\newtheorem*{question*}{Question}

\newtheorem*{bhc}{Boone--Higman conjecture}
\newtheorem*{pbhc}{Permutational Boone--Higman conjecture}

\newtheorem*{thrmAprecise}{Theorem~A (precise version)}

\theoremstyle{definition}

\newtheorem*{definition*}{Definition}
\newtheorem{remark}[theorem]{Remark}

\newtheorem{observation}[theorem]{Observation}

\newcommand{\N}{\mathbb{N}}
\newcommand{\Z}{\mathbb{Z}}

\newcommand{\newword}[1]{\textbf{#1}}

\DeclareMathOperator{\Hom}{Hom}
\DeclareMathOperator{\Aut}{Aut}
\DeclareMathOperator{\MCG}{MCG}
\DeclareMathOperator{\Out}{Out}
\DeclareMathOperator{\Inn}{Inn}

\DeclareMathOperator{\GL}{GL}

\DeclareMathOperator{\F}{\mathcal{F}}

\DeclareMathOperator{\Stab}{Stab}

\newcommand{\Burger}{\Lambda}

\renewcommand{\ll }{\langle\hspace{-.7mm}\langle }
\newcommand{\rr }{\rangle\hspace{-.7mm}\rangle }

\numberwithin{equation}{section}

\begin{document}

\title[Boone--Higman for Aut$(F_n)$ and MCGs of punctured surfaces]{Boone--Higman embeddings of Aut$\boldsymbol{(F_n)}$ and mapping class groups of punctured surfaces}
\date{\today}


\author[J.~Belk]{James Belk}
\address{School of Mathematics \& Statistics, University of Glasgow, Glasgow, Scotland}
\email{jim.belk@glasgow.ac.uk}

\author[F.~Fournier-Facio]{Francesco Fournier-Facio}
\address{Department of Pure Mathematics and Mathematical Statistics, University of Cambridge, UK}
\email{ff373@cam.ac.uk}

\author[J.~Hyde]{James Hyde}
\address{Department of Mathematics and Statistics, Binghamton University, Binghamton, NY}
\email{jhyde1@math.binghamton.edu}

\author[M.~C.~B.~Zaremsky]{Matthew C.~B.~Zaremsky}
\address{Department of Mathematics and Statistics, University at Albany (SUNY), Albany, NY}
\email{mzaremsky@albany.edu}

\begin{abstract}
We prove that the groups $\Aut(F_n)$ satisfy the Boone--Higman conjecture for all $n$, meaning each $\Aut(F_n)$ embeds in a finitely presented simple group. In fact, we prove that each $\Aut(F_n)$ satisfies the ``permutational'' Boone--Higman conjecture, which means the simple group in question can be taken to be a twisted Brin--Thompson group. A far-reaching consequence of our approach is that finitely presented twisted Brin--Thompson groups are universal among finitely presented simple groups that are highly transitive. This is evidence toward the Boone--Higman conjecture being equivalent to its permutational version. Proving the conjecture for $\Aut(F_n)$ also confirms the conjecture for all groups (virtually) embedding into some $\Aut(F_n)$, such as mapping class groups of non-closed surfaces, braid groups, loop braid groups, ribbon braid groups and certain Artin groups. This answers several questions of the first and fourth authors with Bleak and Matucci. Yet another consequence of our approach is that satisfying the permutational Boone--Higman conjecture is closed under free products and regular wreath products, which answers a question of the fourth author, confirms the conjecture for free solvable groups, and reduces the conjecture for finitely generated $3$-manifold groups to the case of non-Seifert fibered closed graph manifolds.
\end{abstract}

\maketitle
\thispagestyle{empty}

\section*{Introduction}

The \newword{Boone--Higman conjecture}, posed by Boone and Higman over 50 years ago \cite{boone74}, predicts the following:

\begin{bhc}
Every finitely generated group with solvable word problem embeds in a finitely presented simple group.
\end{bhc}

Here a group has \newword{solvable word problem} if there is an algorithm to tell whether or not a given word in the generators represents the identity. The converse of the conjecture is easily seen to be true \cite{kuznetsov58}. The Boone--Higman--Thompson theorem \cite{thompson80} states that every finitely generated group with solvable word problem embeds in a finitely generated simple subgroup of a finitely presented group, so the conjecture is that this two-step embedding process can be reduced to an embedding in a single group. Some examples of groups for which the conjecture is known to hold include $\GL_n(\Z)$ \cite{scott84}, hyperbolic groups \cite{bbmz_hyp}, Baumslag--Solitar groups and (finitely generated free)-by-cyclic groups \cite{bux_boonehigman}, contracting or finitely presented self-similar groups \cite{bbmz_hyp,zaremsky_fpss}, and all subgroups thereof. See \cite{bbmz_survey} for more on the history and progress around this conjecture.

Our first main result adds to this list a very prominent and important family of groups in geometric group theory, namely the group $\Aut(F_n)$ of automorphisms of the free group $F_n$, for each $n$.

{\renewcommand*{\thetheorem}{\Alph{theorem}}
\begin{theorem}\label{thrm:autfn}
For each $n$, the group $\Aut(F_n)$ embeds in a finitely presented simple group, and hence satisfies the Boone--Higman conjecture.
\end{theorem}}

This quickly leads to the following corollary, which in particular settles \cite[Problem~5.3(1)]{bbmz_survey}, and partially settles \cite[Problem~5.3(2, 3)]{bbmz_survey}.

{\renewcommand*{\thetheorem}{\Alph{theorem}}
\begin{corollary}\label{cor:autfn_pbhc}
The following groups satisfy the Boone--Higman conjecture:
\begin{itemize}
    \item The (extended) mapping class group of a surface $\Sigma$ of finite type, where $\Sigma$ satisfies any of the following: has non-empty boundary, has at least one puncture, is orientable of genus at most two, is non-orientable of demigenus at most three.\smallskip
    \item The braid groups $\mathcal{B}_n$ for all $n$.\smallskip
    \item Artin groups of types $B_n=C_n$, $D_n$, $I_2(m)$, $\widetilde{A}_n$, $\widetilde{C}_n$, for all $n$ (braid groups are Artin groups of type $A_n$), and $H_3$. \smallskip
    \item Loop braid groups $\mathcal{LB}_n$, extended loop braid groups $\mathcal{LB}_n^{ext}$, ribbon braid groups $\mathcal{RB}_n$.
\end{itemize}
\end{corollary}}

For spherical and Euclidean Artin groups, the Boone--Higman conjecture remains open for the exceptional type spherical Artin groups other than type $H_3$, and the Euclidean Artin groups other than type $\widetilde{A}_n$ and $\widetilde{C}_n$; in rank $3$ it is also known for type $\widetilde{G}_2$ (and $\widetilde{A}_2$ and $\widetilde{C}_2$, which we recover) \cite{bux_boonehigman}. Note that for general Artin groups it is a famous open problem whether they even have solvable word problem \cite[Problem~10]{charney:problems}, but for all spherical and Euclidean types they do \cite{deligne:Kpi1, mccammond17}. Also note that all Coxeter groups satisfy the Boone--Higman conjecture \cite{belk20}.

For mapping class groups, the conjecture remains open for closed surfaces of genus at least three. Note that the Boone--Higman conjecture is known for many finitely generated linear groups \cite{scott84, zaremsky_fpss}, but it is an open problem whether mapping class groups of closed surfaces of genus at least three are linear. The case of genus two is special, thanks to an exceptional relationship with braid groups, which is behind both the proof of linearity in \cite{bigelow:budney} and our proof of the Boone--Higman conjecture for these groups.

As a further application of Theorem~\ref{thrm:autfn}, we also recover the conjecture for some important classes of groups for which it was already known, such as virtually compact special groups and finitely generated virtually free-by-cyclic groups \cite[Problem~5.3(8)]{bbmz_survey}; see Remark~\ref{rem:othergroups}. Our approach to finitely generated free-by-cyclic groups also easily adapts to more general extensions: see Proposition~\ref{prop:extensions}. In particular it covers finitely generated surface-by-cyclic groups, which reduces the (still open) conjecture for finitely generated $3$-manifold groups to the non-Seifert fibered closed graph manifold case; see Corollary~\ref{cor:graph:mfld}.

\medskip

In fact we prove a more precise version of Theorem~\ref{thrm:autfn} that is stronger in two ways. Before stating the strong version we need some background. First, we work with the much larger group $\Aut_G(G*F_n)$ of $G$-automorphisms of the free product $G*F_n$, for $G$ a group with certain properties. Here a \newword{$\boldsymbol{G}$-automorphism} of $G*F_n$ is one that restricts to the identity on $G$, so clearly $\Aut(F_n)$ embeds into $\Aut_G(G*F_n)$ for any $G$. Second, the finitely presented simple groups we obtain in Theorem~\ref{thrm:autfn} always come from the family of twisted Brin--Thompson groups, thus proving that the groups in question satisfy the \emph{a priori} stronger ``permutational'' form of the Boone--Higman conjecture, which will be explained shortly. Twisted Brin--Thompson groups are defined as follows. Given a group $\Gamma$ acting faithfully on a set $S$, we can form the permutational wreath product $V\wr_S\Gamma$, where $V$ is the classical Thompson group. This admits a natural action on $(\{0,1\}^\N)^S$, and the \newword{twisted Brin--Thompson group} $SV_\Gamma$ is the corresponding topological full group. These groups were introduced by the first and fourth authors in \cite{belk22}, following the ``untwisted'' case due to Brin \cite{brin04}; also see \cite{zar_taste,bbmz_survey} for more details.

\begin{thrmAprecise}
For each $n\ge 2$, the group $\Aut(F_n)$ (and hence all of its subgroups) embeds in a finitely presented simple group, namely the twisted Brin--Thompson group $SV_\Gamma$ where $\Gamma=\Aut_V(V*F_n)$ and $S=\Hom_V(V*F_n,V)$, and hence satisfies the Boone--Higman conjecture.
\end{thrmAprecise}

The choice of Thompson's group $V$ for the group $G$ is just a convenient one for a concise statement, however, as we will see below, any finitely presented simple MIF group will do. Here $\Hom_G(G*F_n,G)$ denotes the set of $G$-homomorphisms from $G*F_n$ to $G$, that is, homomorphisms restricting to the identity on $G$. Equivalently, it is the set of retractions of $G*F_n$ onto the free factor $G$. The action of $\Aut_G(G*F_n)$ on $\Hom_G(G*F_n,G)$ is by precomposition. Note that elements of $\Hom_G(G*F_n,G)$ are in one-to-one correspondence with ordered $n$-tuples of elements of~$G$, so the set $S$ in Theorem~\ref{thrm:autfn} can alternatively be viewed as~$V^n$, where again $V$ is the classical Thompson group.

As we mentioned, this shows that our groups satisfy not only the Boone--Higman conjecture but even the ``permutational'' Boone--Higman conjecture. Let us recall what this means. Given a group $\Gamma$ acting faithfully on a set $S$, we say that the action is of \newword{type~(A)} if:
\begin{enumerate}
    \item The group $\Gamma$ is finitely presented.\smallskip
    \item The stabilizer $\Stab_\Gamma(s)$ is finitely generated for all $s\in S$.\smallskip
    \item The action has finitely many orbits of pairs, i.e., the diagonal action of $\Gamma$ on $S\times S$ has finitely many orbits.
\end{enumerate}
Twisted Brin--Thompson groups are always simple \cite{belk22}, and it turns out that $SV_\Gamma$ is finitely presented if and only if the action of $\Gamma$ on $S$ is of type~(A) \cite{zaremsky_fp_tbt}; this is analogous to the fact that $V\wr_S \Gamma$ is finitely presented if and only if the action is of type~(A) \cite{cornulier06}. Now the \newword{permutational Boone--Higman conjecture} \cite{zaremsky_fp_tbt} predicts:

\begin{pbhc}
Every finitely generated group with solvable word problem embeds in a group admitting an action of type~(A), and hence in a finitely presented (simple) twisted Brin--Thompson group.
\end{pbhc}

At this point we should clarify some terminology. When we say that a finitely generated group ``satisfies the Boone--Higman conjecture'' we are implicitly saying it has solvable word problem, and explicitly saying it embeds in a finitely presented simple group. Similarly, when we say it ``satisfies the permutational Boone--Higman conjecture'' we are implicitly saying it has solvable word problem, and explicitly saying it embeds in a group admitting an action of type~(A). For brevity, we will often abbreviate ``permutational Boone--Higman'' by ``PBH'', and ``Boone--Higman'' by ``BH'', and write things like, ``satisfies PBH,'' for, ``satisfies the PBH conjecture.'' Writing ``(P)'' in parentheses emphasizes that even satisfying BH is a new result. By a \newword{Boone--Higman embedding}, we mean an embedding of a group as a subgroup of a finitely presented simple group.

\medskip

In addition to proving (P)BH for $\Aut(F_n)$ and all its subgroups, our approach has some interesting consequences concerning the relationship between BH and PBH. A fundamental question in this context is whether BH and PBH are equivalent \cite[Question 1.6]{bbmz_hyp}, which amounts to determining whether every finitely presented simple group satisfies PBH, and our work here is a step in the direction of a positive answer. Namely, we establish that at least this holds provided the finitely presented simple group is highly transitive or, more generally, mixed identity-free. These conditions are defined as follows:
\begin{itemize}
    \item A group $G$ is \newword{highly transitive} if it admits a faithful action on a set that is $k$-transitive for every $k \geq 1$.  For example, Thompson's group $V$ is highly transitive because of its action on the orbit of $\overline{0}$ in $\{0,1\}^\N$.  Similarly, any twisted Brin--Thompson group is highly transitive \cite[Proposition~5.4]{bbmz_hyp}.\smallskip
    \item If $G$ is a group, a non-trivial word $w(x_1,\ldots,x_n)\in G*F_n$ is called a \newword{mixed identity} in~$G$ if $w(g_1,\ldots,g_n)=1$ for all $g_1,\ldots,g_n\in G$.  That is, $w \in G * F_n$ is a mixed identity if it belongs to the kernel of every element of $\mathrm{Hom}_G(G * F_n, G)$, which explains the relevance to our construction.  A group $G$ is \newword{mixed identity-free (MIF)} if it has no mixed identities.  
    MIF groups are always lawless, but the converse does not hold; for example, Thompson's group $T$ is lawless but not~MIF \cite[Proposition~4.7]{leboudec22}.
\end{itemize}
See \cite{anashin} for background on mixed identities, and \cite[Theorem~5.9]{hull16} or \cite[Proposition~A.1]{leboudec22} for a proof that every finitely generated, highly transitive simple group is~MIF.

Our main result in this context is the following.

{\renewcommand*{\thetheorem}{\Alph{theorem}}
\begin{theorem}\label{thrm:mif_pbhc}
For a finitely generated group $G$ (with solvable word problem) the following are equivalent:
\begin{enumerate}
    \item $G$ satisfies the permutational Boone--Higman conjecture, i.e., it embeds in a group admitting an action of type~(A).\smallskip
    \item $G$ embeds in a finitely presented (simple) twisted Brin--Thompson group.\smallskip
    \item $G$ embeds in a finitely presented simple group that is highly transitive.\smallskip
    \item $G$ embeds in a finitely presented simple group that is MIF.
\end{enumerate}
In particular, every finitely presented simple group that is highly transitive, or more generally MIF, satisfies the permutational Boone--Higman conjecture, as does every subgroup thereof.
\end{theorem}}

The new part of Theorem~\ref{thrm:mif_pbhc} is the implication (iv) $\Rightarrow$ (i). In fact, twisted Brin--Thompson groups are themselves highly transitive and thus MIF, so Theorem~\ref{thrm:mif_pbhc} has the following immediate consequence.

{\renewcommand*{\thetheorem}{\Alph{theorem}}
\begin{corollary}\label{cor:universal}
Finitely presented (simple) twisted Brin--Thompson groups are universal among finitely presented simple highly transitive groups, and more generally among finitely presented simple MIF groups.\qed
\end{corollary}}

What we mean by ``universal'' is that within the class of finitely presented simple highly transitive groups, every member of this class embeds in a member of the subclass of finitely presented (simple) twisted Brin--Thompson groups, and similarly for MIF. This emphasizes the outsized importance of twisted Brin--Thompson groups in the study of finitely presented simple groups. See Remark~\ref{rmk:mif} for a discussion of which known finitely presented simple groups are highly transitive and MIF, including the interesting case of Burger--Mozes groups.

\medskip

The bulk of the work in proving Theorems~\ref{thrm:autfn} and~\ref{thrm:mif_pbhc} amounts to proving the (P)BH conjecture for certain groups of the form $\Aut_G(G*F_n)$, detailed in the following:

{\renewcommand*{\thetheorem}{\Alph{theorem}}
\begin{theorem}\label{thrm:autg_pbhc}
For any finitely presented simple MIF group $G$, the group $\Aut_G(G*F_n)$ ($n\ge 2$) admits an action of type~(A), and hence satisfies the (permutational) Boone--Higman conjecture.
\end{theorem}}

Theorem~\ref{thrm:autg_pbhc} is the key to proving that (iv) implies (i) in Theorem~\ref{thrm:mif_pbhc}, and also to proving Theorem~\ref{thrm:autfn}. It also has other important consequences, relating to permanence properties of the PBH conjecture. Indeed, the main advantage of knowing PBH for a group, instead of just BH, is that PBH is stable under several group-theoretic constructions. For instance, PBH is stable under commensurability and direct products \cite{zaremsky_fp_tbt}. Now Theorem~\ref{thrm:autg_pbhc} adds two more such properties to the list, namely free products and regular wreath products, answering \cite[Question~5.7]{zaremsky_fp_tbt}.

{\renewcommand*{\thetheorem}{\Alph{theorem}}
\begin{corollary}\label{cor:freeprod_pbhc}
If $A$ and $B$ satisfy the permutational Boone--Higman conjecture, then the free product $A*B$ also satisfies the (permutational) Boone--Higman conjecture.
\end{corollary}}

For example the free product $V*V$ of Thompson's group $V$ with itself satisfies (P)BH, which was not known before. Much more generally, any free product of twisted Brin--Thompson groups satisfies the conjecture. This can also be seen as a step towards establishing the conjecture for finitely generated $3$-manifold groups: indeed, thanks to powerful structural results on such groups \cite{mfld}, we can reduce the problem to non-Seifert fibered closed graph manifolds (see Corollary~\ref{cor:graph:mfld}).

{\renewcommand*{\thetheorem}{\Alph{theorem}}
\begin{corollary}\label{cor:wreath_pbhc}
If $A$ and $B$ satisfy the permutational Boone--Higman conjecture, then the regular wreath product $A \wr B$ also satisfies the (permutational) Boone--Higman conjecture.
\end{corollary}}

This applies to the following important class.

{\renewcommand*{\thetheorem}{\Alph{theorem}}
\begin{corollary}\label{cor:freesolvable}
Finitely generated free solvable groups satisfy the (permutational) Boone--Higman conjecture.
\end{corollary}}

Note that not all finitely generated solvable groups have solvable word problem \cite{harlampovic81}, and even for those that do, the Boone--Higman conjecture remains open in general, for example for finitely generated metabelian groups \cite[Problem~5.3(7)]{bbmz_survey}. A further stability property that covers some extensions is recorded in Proposition~\ref{prop:extensions}.

\medskip

The key to all of this is Theorem~\ref{thrm:autg_pbhc}. To prove this, we consider the action of $\Aut_G(G*F_n)$ on $\Hom_G(G*F_n,G)$, for $G$ a finitely presented simple MIF group, and prove it is of type~(A). We should mention that in the course of the proof that the action is of type~(A), the only part that actually requires $G$ to be MIF is the proof of faithfulness (indeed, this construction is used in \cite{abstract} to show that every finitely presented simple group embeds into a finitely presented ``relatively simple'' group from the Thompson world). If the MIF requirement could be dropped, for example by proving that the quotient of $\Aut_G(G*F_n)$ by the kernel of the action on $\Hom_G(G*F_n,G)$ is finitely presented, or by proving that every finitely presented simple group embeds in one that is MIF, then we would conclude that finitely presented (simple) twisted Brin--Thompson groups are universal among \emph{all} finitely presented simple groups, and that BH is equivalent to PBH. We do not know whether every finitely presented simple group embeds in a finitely presented simple MIF group; there do exist finitely presented simple groups that are not MIF, for example among groups acting on the circle, see \cite[Proposition~4.7]{leboudec22}, such as Thompson's group $T$ and certain variants, including torsion-free examples \cite{hyde25}. Note however that $T$ embeds in Thompson's group $V$, which is MIF, so it is unclear what to expect in general.

\medskip

This paper is organized as follows. In Section~\ref{sec:group} we set up the group $\Aut_G(G*F_n)$ and prove some basic properties of it. In Section~\ref{sec:action} we look at the action of $\Aut_G(G*F_n)$ on $\Hom_G(G*F_n,G)$ and prove that if $G$ is nice then the action is nice. In Section~\ref{sec:main_proof} we prove our main results. Section~\ref{sec:permanence} is devoted to permanence properties, and in Section~\ref{sec:AutFn} we prove Corollary~\ref{cor:autfn_pbhc}.

\subsection*{Acknowledgments} We thank Ian Agol, Collin Bleak, Matt Brin, Benjamin Br\"uck, \mbox{Kai-Uwe} Bux, Corentin Le Bars, Adrien Le Boudec, Yash Lodha, Robbie Lyman, Dan Margalit, Timoth{\'e}e Marquis, Nicol{\'a}s Matte Bon, Francesco Matucci, Denis Osin, Mark Pengitore, Andy Putman, Ekaterina Rybak, Ignat Soroko, Ric Wade, Henry Wilton, Becca Winarski, and Xiaolei Wu for helpful conversations and pointers to references. We also thank the anonymous referees for helpful comments and suggestions. FFF is supported by the Herchel Smith Postdoctoral Fellowship Fund.

\section{Groups of \texorpdfstring{$G$}{}-automorphisms}\label{sec:group}

We introduce some notions from algebraic geometry on groups, as developed in \cite{AGG}; these will be at the heart of our constructions.
Let $G$ be a group. A \newword{$\boldsymbol{G}$-group} is a group containing a designated copy of $G$, which we usually identify with $G$. A \newword{$\boldsymbol{G}$-homomorphism} from a $G$-group $H$ to another $G$-group $H'$ is a group homomorphism $\phi\colon H\to H'$ such that the restriction $\phi|_G$ is the identity; if $\phi$ is bijective then it is a \newword{$\boldsymbol{G}$-isomorphism}. A \newword{$\boldsymbol{G}$-automorphism} is a $G$-isomorphism from a $G$-group to itself; the $G$-automorphisms of a $G$-group $H$ form the \newword{group of $\boldsymbol{G}$-automorphisms} of $H$, denoted $\Aut_G(H)$. The \newword{free $\boldsymbol{G}$-group of rank $n$} is the free product $G*F_n$, where $F_n$ is the free group of rank~$n$. Write $\Hom_G(H,H')$ for the set of all $G$-homomorphisms from a $G$-group $H$ to another $G$-group~$H'$. A \newword{$\boldsymbol{G}$-presentation} of a $G$-group $H$ is a surjective $G$-homomorphism from $G*F_n$ to~$H$. In particular, $\Hom_G(G*F_n,G)$ is the set of all $G$-presentations of $G$ on $n$ generators.

\begin{proposition}\label{prop:fp}
If $G$ is finitely presented and has trivial center, then $\Aut_G(G*F_n)$ is finitely presented.
\end{proposition}

\begin{proof}
If $G=\{1\}$ then $\Aut_G(G*F_n)=\Aut(F_n)$ is finitely presented. Now suppose $G\ne\{1\}$. Since $G$ has trivial center, $\Aut_G(G*F_n)$ contains no inner automorphisms of $G*F_n$ except the trivial one, and so the subgroup $\Inn(G*F_n) \cdot \Aut_G(G*F_n)$ of $\Aut(G*F_n)$ splits as a semidirect product $\Inn(G*F_n)\rtimes \Aut_G(G*F_n)$. Because $\Inn(G*F_n)$ is finitely generated, in order to prove that $\Aut_G(G*F_n)$ is finitely presented, it now suffices to prove that $\Inn(G*F_n)\rtimes \Aut_G(G*F_n)$ is finitely presented. This is the group of all automorphisms of $G*F_n$ that agree with an inner automorphism on $G$, and this group is shown to be finitely presented in \cite[Theorem~4.1]{carette11} (where it is denoted $\Aut_i(G*F_n)$), again using the fact that $G$ has trivial center.
\end{proof}

It will be convenient later to have an explicit generating set $X$ for $\Aut_G(G*F_n)$ pinned down. At this point we assume $G$ is freely indecomposable with trivial center (for example our main case of interest: $G$ is non-abelian simple). The specific generating set we will use is far from optimal, but the way we lay it out here will be useful. First let $Y$ be the set of all elements of $\Aut_G(G*F_n)$ consisting of the following automorphisms (where in all cases, of course, $G$ is fixed):
\begin{enumerate}
    \item Those automorphisms induced by permutations of the basis $\{x_1,\dots,x_n\}$.\smallskip
    \item The automorphism fixing $x_2,\dots,x_n$ and sending $x_1$ to $x_1^{-1}$.\smallskip
    \item The automorphisms fixing $x_2,\dots,x_n$ and sending $x_1$ to $g^{-1}x_1 g$ for some $g\in G$.\smallskip
    \item The automorphisms fixing $x_2,\dots,x_n$ and sending $x_1$ to $x_1 x_2$ or $x_1 x_2^{-1}$.
\end{enumerate}
Write $X$ for the subset of $\Aut_G(G*F_n)$ that consists of $Y$ along with:
\begin{enumerate}\setcounter{enumi}{4}
    \item The automorphisms fixing $x_2,\dots,x_n$ and sending $x_1$ to $x_1 g$ for some $g\in G$.
\end{enumerate}
Note that $X$ and $Y$ are symmetric, i.e., inverse-closed.

\begin{proposition}\label{prop:gens}
If $G$ is freely indecomposable with trivial center, then $X$ generates $\Aut_G(G*F_n)$.
\end{proposition}

\begin{proof}
Since all elements of $X$ belong to $\Aut_G(G*F_n)$, it suffices to show that $\langle X \cup \Inn(G*F_n) \rangle$ equals $A \coloneqq \Inn(G*F_n) \rtimes \Aut_G(G*F_n)$ (recall from the proof of Proposition~\ref{prop:fp} that this is indeed a semidirect product, because $G$ has trivial center). In \cite[Section~4]{carette11} a generating set for $A$ is given, consisting of two types of automorphisms, so it suffices to show that each of these is contained in $\langle X \cup \Inn(G*F_n) \rangle$. Let us start by noticing that $\langle X \rangle$ contains $\Aut(F_n)$: indeed the generators of type (i), (ii) and (iv) are the standard Nielsen generators of $\Aut(F_n)$.

The first type is the set of \newword{Whitehead automorphisms}, in the sense of \cite[Section~1]{gilbert87}. These are those automorphisms $\alpha \in \Aut(G*F_n)$ with the following property: there exists an element $y \in G \cup \{x_1^{\pm}, \ldots, x_n^{\pm} \}$ such that $\alpha|_G$ is either the identity or conjugation by $y$, and each $x_i$ is mapped to one of $x_i$, $x_i y$, $y^{-1} x_i$, or $y^{-1} x_i y$. In case $\alpha|_G$ is the identity, this is either a composition of generators of types (i), (iii), and (v) (if $y \in G$) or an element of $\Aut(F_n)$ (otherwise), so it belongs to $\langle X \rangle$. In case $\alpha|_G$ is conjugation by $y$, it is a composition of the inner automorphism defined by $y$ with either a composition of generators of types (i), (iii), and (v) (if $y \in G$) or an element of $\Aut(F_n)$ (otherwise). In all cases, we see that every Whitehead automorphism belongs to $\langle X \cup \Inn(G*F_n) \rangle$.

The second type is the set of automorphisms $\alpha$ that act on $G$ as conjugation by an element of $G$, and send each $x_i$ to $x_i^{\pm 1}$. This is a product of a Whitehead automorphism (conjugation on $G$ and identity on $F_n$) and an element of $\Aut(F_n)$ (inverting the appropriate~$x_i$). Therefore generators of this type also belong to $\langle X \cup \Inn(G*F_n) \rangle$ and we are done.
\end{proof}

Note that $\Aut_G(G*F_n)$ contains $\Aut(F_n)$, thanks to the generators (i), (ii), and (iv), and also contains $G$, thanks to the generators (v). Thus, in order to prove the (P)BH conjecture for $\Aut(F_n)$ it suffices to prove it for $\Aut_G(G*F_n)$ for some $G$, and in order to prove it for a particular $G$ it suffices to prove it for $\Aut_G(G*F_n)$ for some $n$.

\section{The action}\label{sec:action}

In this section we analyze the (right) action of $\Aut_G(G*F_n)$ on $\Hom_G(G*F_n,G)$ by precomposition.

\begin{lemma}\label{lem:trans}
The action of $\Aut_G(G*F_n)$ on $\Hom_G(G*F_n,G)$ is transitive.
\end{lemma}

\begin{proof}
Let $\phi,\psi\in \Hom_G(G*F_n,G)$. Let $\alpha$ be the automorphism of $G*F_n$ that fixes $G$ and sends each $x_i$ to $\psi(x_i)(\phi(x_i)^{-1})x_i$. Then $\phi\circ\alpha=\psi$.
\end{proof}

Recall that a group $G$ is called \newword{mixed identity-free (MIF)} if for all $n$ and all $1\ne w\in G*F_n$, there exists a $G$-homomorphism $\phi\colon G*F_n\to G$ such that $\phi(w)\ne 1$. For $G$ a non-MIF group, any $1\ne w$ lying in the kernel of every such $G$-homomorphism is called a \newword{mixed identity}, also called a \newword{law with constants}. The mixed identities in $G*F_n$ comprise the \newword{$\boldsymbol{G}$-Jacobson radical} $J_G(G*F_n)$ of $G*F_n$; see \cite{AGG}. Thus, $G$ is MIF if and only if $J_G(G*F_n)=\{1\}$ for all $n$, which is equivalent to $J_G(G*F_1)=\{1\}$ \cite[Remark~5.1]{hull16}.

\begin{lemma}\label{lem:faithful}
Suppose that $G$ is MIF. Then the action of $\Aut_G(G*F_n)$ on $\Hom_G(G*F_n,G)$ is faithful.
\end{lemma}

We reiterate that this is the only place in all the analysis of the action where we need $G$ to be MIF.

\begin{proof}
The stabilizer in $\Aut_G(G*F_n)$ of $\phi\in \Hom_G(G*F_n,G)$ consists of all $\alpha$ such that $\phi\circ\alpha=\phi$. Let $\alpha$ be a non-trivial element of $\Aut_G(G*F_n)$, so we must find an element of $\Hom_G(G*F_n,G)$ such that $\alpha$ is not in its stabilizer. Since $\alpha$ is non-trivial, there exists a basis element $x_i$ of $F_n$ such that $\alpha(x_i)\ne x_i$. Since $G$ is MIF, there exists $\phi\in \Hom_G(G*F_n,G)$ sending $x_i^{-1}\alpha(x_i)$ to a non-identity element. Now $\phi\circ\alpha$ sends $x_i$ to $\phi(\alpha(x_i))\ne \phi(x_i)$, so $\phi \circ \alpha \neq \phi$ which shows that $\alpha$ acts non-trivially on $\Hom_G(G*F_n,G)$.
\end{proof}

\begin{proposition}\label{prop:2trans}
Suppose $G$ is an infinite simple group, and $n\ge 2$. Then the action of $\Aut_G(G*F_n)$ on $\Hom_G(G*F_n,G)$ is highly transitive.
\end{proposition}

In practice we will only need $2$-transitivity, in order to deduce that the action is of type (A). This is the only place in the paper where high transitivity plays this role; in the following sections we will only see high transitivity as part of a sufficient condition for being MIF.

\begin{proof}
Before embarking on the proof of high transitivity, let us establish the existence of certain elements of $G*\langle x\rangle$. Given a subset $T\subseteq G$ and an element $g\in G\setminus T$, say that a word $w(x)\in G*\langle x\rangle$ \newword{separates} $T$ from $g$ if $w(t)=1$ for all $t\in T$ and $w(g)\ne 1$. We claim that if $G$ is infinite simple then we can separate any finite $T\subseteq G$ from any $g\in G\setminus T$. Say $T=\{t_1,\dots,t_k\}$. Let $w_1(x)\in G*\langle x\rangle$ be $w_1(x)\coloneqq t_1^{-1}x$, so $w_1(t_1)=1$ and $w_1(g)\ne 1$. Since $t_2^{-1}g$ is non-trivial and $G$ is simple, $G$ is generated by the conjugates of $t_2^{-1}g$. Since $w_1(g)$ is non-trivial and $G$ has trivial center, there must exist a conjugate of $t_2^{-1}g$ that does not commute with $w_1(g)$, say $(t_2^{-1}g)^{h_2}$ for $h_2\in G$. (Here $x^y$ denotes $y^{-1}xy$.) Let $w_2(x)\coloneqq [w_1(x),(t_2^{-1}x)^{h_2}]$, so $w_2(t_1)=w_2(t_2)=1$, and $w_2(g)\ne 1$. Continuing this way, we get a sequence of words $w_1,w_2,\dots,w_k\in G*\langle x\rangle$ defined by $w_i(x)\coloneqq [w_{i-1}(x),(t_i^{-1}x)^{h_i}]$ for appropriate $h_i\in G$, such that $w_i(t_j)=1$ for all $1\le j\le i$ and $w_i(g)\ne 1$. In particular, $w_k(x)$ is our desired element that separates $T$ from $g$.

Now we will prove that our action is highly transitive. It suffices to prove that for all $m\ge 1$ we can choose some $m$-element subset $\Phi$ of $\Hom_G(G*F_n,G)$ such that the pointwise fixer of $\Phi$ is transitive on the complement of $\Phi$; the result will then follow by induction, with Lemma~\ref{lem:trans} as the base case. Let us choose $\Phi=\{\phi_1,\dots,\phi_m\}$ such that for each $1\le i\le m$ we have $\phi_i(x_j)=1$ for all $2\le j\le n$. Write $t_i=\phi_i(x_1)$ for each $i$ and let $T = \{ t_1, \ldots, t_m\}$. Fix some $1\ne z\in G$ for the duration of the proof. Now let $\psi\in \Hom_G(G*F_n,G)\setminus \Phi$ be arbitrary, and our goal is to find an element of $\Aut_G(G*F_n)$ fixing $\Phi$ and taking $\psi$ to the map sending $x_2$ to $z$ and all other generators to $1$, thus showing that there is a single orbit. First suppose $\psi(x_1) \in T$. Since $\phi_i(x_j)=1$ for all $1\le i\le m$ and $2\le j\le n$, but $\psi\notin\Phi$, we know $\psi(x_j)\ne 1$ for some $2\le j\le n$. Let $\alpha \in \Aut_G(G*F_n)$ fix $x_2,\dots,x_n$ and send $x_1$ to $x_1 u(x_j)$ where $u(x_j)$ is a product of $G$-conjugates of $x_j$ and $x_j^{-1}$, chosen in such a way that $\psi(x_1 u(x_j)) \notin T$ -- this is possible because $G$ is simple and $\psi(x_j)\ne 1$. Thus at this point, up to the action, we may assume that $\psi(x_1) \notin T$. Write $g=\psi(x_1)$, so $g \in G \setminus T$. By the first paragraph we can choose $w(x_1)\in G*\langle x_1\rangle$ that separates $T$ from $g$, so $w(t_i)=1$ for all $1\le i\le m$ and $w(g)\ne 1$. Since $w(g)\ne 1$ and $G$ is simple, we can write the element $\psi(x_2)^{-1}z$ as a product of $G$-conjugates of $w(g)$ and $w(g)^{-1}$, say $\psi(x_2)^{-1}z = w(g)^{\pm h_1}\cdots w(g)^{\pm h_\ell}$. (Here we write $x^{-y}$ for $(x^{-1})^y$.) Now consider the automorphism $\alpha$ in $\Aut_G(G*F_n)$ that fixes $x_1,x_3,\dots,x_n$ and sends $x_2$ to $x_2 w(x_1)^{\pm h_1}\cdots w(x_1)^{\pm h_\ell}$. Note that $\alpha$ fixes $\Phi$ since $w(t_i)=1$ for all $1\le i\le m$, and $\psi(\alpha(x_2))=z$ since $\psi(x_2) w(g)^{\pm h_1}\cdots w(g)^{\pm h_\ell} = z$, so at this point thanks to the action of $\Aut_G(G*F_n)$ we are free to assume without loss of generality that $\psi(x_2)=z$. Using a similar trick, now with automorphisms that fix $x_2$ and multiply all the other $x_i$ by appropriate words in $x_2$ (leveraging that $z\ne 1$ and $G$ is simple), since $\phi_i(x_2)=1$ for all $1\le i\le m$ we can fix $\Phi$ while taking $\psi$ to the map sending $x_2$ to $z$ and all other generators to $1$, as desired.
\end{proof}

To reiterate, in aiming for the action of $\Aut_G(G*F_n)$ on $\Hom_G(G*F_n,G)$ to be of type~(A), the only place where we need our simple group $G$ to be MIF is in the proof of faithfulness in Lemma~\ref{lem:faithful}, so it is worth mentioning what happens in the non-MIF case.

\begin{observation}\label{obs:nonmif}
When $n=1$ and $G\ne\{1\}$ is freely indecomposable with trivial center, the action of $\Aut_G(G*F_1)$ on $\Hom_G(G*F_1,G)$ is faithful. When $n\ge 2$ and $G$ is not MIF, the action is never faithful.
\end{observation}

\begin{proof}
The kernel of the action consists of all automorphisms $\alpha\in\Aut_G(G*F_n)$ such that $\alpha$ sends each $x_i$ to itself times a mixed identity. When $n=1$, any $\alpha\in\Aut_G(G*F_1)$ sends $x_1$ to an element of the form $g x_1 h$ or $gx_1^{-1} h$ for $g,h\in G$ (this is clear from looking at the generating set $X$ from Proposition~\ref{prop:gens}). Writing this as $x_1(x_1^{-1}gx_1^{\pm 1} h)$, for $\alpha$ to lie in the kernel we would need $x_1^{-1}gx_1 h$ or $x_1^{-1}gx_1^{-1} h$ to be a mixed identity. In either case, evaluating at $x_1=1$ shows that $h=g^{-1}$. In the first case this means $g$ is central, hence $g=1$ and $\alpha$ is the identity. In the second case this means conjugation by $g$ inverts every element of $G$, so inversion is a homomorphism and hence $G$ is abelian, contradicting that $G$ is non-trivial with trivial center.

Now let $n\ge 2$ and suppose $G$ is not MIF. By \cite[Remark~5.1]{hull16}, $J_G(G*\langle x_2\rangle)\ne\{1\}$, and so the kernel contains the non-trivial automorphism fixing $x_2,\dots,x_n$ and sending $x_1$ to $x_1 w$ for $1\ne w\in J_G(G*\langle x_2\rangle)$.
\end{proof}

Even though we do have faithfulness in the $n=1$ case regardless of whether $G$ is MIF, we should emphasize that the proof of Proposition~\ref{prop:2trans} completely breaks down when $n=1$ and we very likely have infinitely many orbits of pairs. Indeed, when $n=1$ automorphisms that stabilize the homomorphism sending $x_1$ to $1$ can only send $x_1$ to $x_1^{\pm g}$ for $g\in G$, so we would need $G$ to have finitely many conjugacy classes: this is certainly not always the case (and it is an open question whether it can ever occur for a finitely presented infinite $G$).

\begin{remark}\label{rmk:garion}
Our focus here is the action of $\Aut_G(G*F_n)$ on $\Hom_G(G*F_n, G)$. This is, at least in spirit, closely related to the action of $\Aut(F_n)$ on the subset of $\Hom(F_n, G)$ consisting of surjective homomorphisms (in words, the action of $\Aut(F_n)$ on presentations of $G$ with $n$ generators), which we denote by $P_n(G)$. This is a very important action, whose study was popularized by Wiegold especially in the case of finite simple groups, see e.g., the excellent surveys \cite{pak, lubotzky}; but it has interesting features also in the case of finitely generated infinite simple groups \cite{wiegold:wilson, wiegold:2gen, characteristic}.

The action of $\Aut(F_n)$ on $P_n(G)$ is by precomposition. Quotienting out $P_n(G)$ by the natural $\Aut(G)$ action by postcomposition defines a set $\overline{P}_n(G)$ on which $\Inn(F_n)$ acts trivially, and hence on which $\Out(F_n)$ acts. It turns out that a sufficient condition for this action to be highly transitive is that $G$ be a lawless \newword{Tarski monster} (i.e., a lawless group all of whose proper subgroups are cyclic) \cite{garion13}. While being lawless is much less restrictive than being MIF (for example every group with a non-abelian free subgroup is lawless), the requirement that $G$ be a Tarski monster is extremely restrictive, so much so that no finitely presented example is known. Therefore it is unclear whether the stabilizers of such actions can ever be finitely generated, and so this does not seem to produce type~(A) actions for $\Out(F_n)$. It is more likely that one could produce an \emph{embedding} of $\Out(F_n)$ in a group with a type~(A) action, but this remains open~\cite[Problem~5.3]{bbmz_survey}.

It is now known that all acylindrically hyperbolic groups admit faithful highly transitive actions \cite{hull16}, and this applies in particular to $\Out(F_n)$ \cite{bestvinafeighn} (and, for that matter, to $\Aut(F_n)$ \cite{genevoishorbez}). But in this more general result the action is built using small cancellation techniques, and so it is even less likely that the methods could lead to a type~(A) action for $\Out(F_n)$.
\end{remark}

\subsection{Finite generation of the stabilizer}\label{ssec:fg_stab}

In this subsection we prove that all point stabilizers in $\Aut_G(G*F_n)$ of elements of $\Hom_G(G*F_n,G)$ are finitely generated. Since the action is transitive by Lemma~\ref{lem:trans}, it is enough to do this for a single point. Let $\phi_0\in \Hom_G(G*F_n,G)$ be the $G$-homomorphism $\phi_0\colon G*F_n\to G$ sending every $x_i$ to $1$, and write $\Stab(\phi_0)$ for $\Stab_{\Aut_G(G*F_n)}(\phi_0)$. Note that $\Stab(\phi_0)$ consists of all $G$-automorphisms of $G*F_n$ that send each $x_i$ to an element of the normal closure of $\{x_1,\dots,x_n\}$ in $G*F_n$. For example, all of $\Aut(F_n)$ lies in $\Stab(\phi_0)$.

\begin{proposition}\label{prop:fg_stab}
Suppose $G$ is freely indecomposable with trivial center. If $G$ is finitely generated then $\Stab(\phi_0)$ is finitely generated.
\end{proposition}

\begin{proof}
We have $G*F_n = \ll F_n \rr \rtimes G$, so, viewing this as an ``external'' semidirect product, there is a function (not homomorphism) $t\colon G*F_n \to \ll F_n\rr$ sending $(f,g)$ to $f$. One can view this as the function describing the action of $G$ on $\ll F_n\rr$. We will also view $t$ as the function $(f,g)\mapsto(f,1)$ from $G*F_n$ to itself. Note that $t$ is invariant under right multiplication by elements of $G$, and is equivariant with respect to conjugation by elements of $G$. Also note that since $(f,g)^{-1}=(g^{-1}f^{-1}g,g^{-1})$, we have
\begin{equation}
\label{eq:t:inverse}
    t((f,g)^{-1}) = g^{-1} t(f,g)^{-1} g,
\end{equation}
that is, $t$ respects inversion up to a straightforward conjugation.

Let $T$ be the function (not homomorphism) from $\Aut_G(G*F_n)$ to $\Stab(\phi_0)$ that sends $\alpha$ to the homomorphism $T(\alpha)\colon G*F_n\to G*F_n$ that (fixes $G$ and) sends each $x_i$ to $T(\alpha)(x_i)\coloneqq t(\alpha(x_i))$. For each $1\le i\le n$ we have $T(\alpha)(x_i)=\alpha(x_i)g_i=\alpha(x_i g_i)$ for some $g_i\in G$, so letting $\beta_\alpha\in\Aut_G(G*F_n)$ be the $G$-automorphism sending each $x_i$ to $x_i g_i$, we get $T(\alpha)=\alpha\beta_\alpha$. In particular, $T(\alpha)$ is a $G$-automorphism. It clearly stabilizes $\ll F_n\rr$, hence is an element of $\Stab(\phi_0)$. Also note that if $\alpha\in\Stab(\phi_0)$ then $T(\alpha)=\alpha$, since $t$ restricted to $\ll F_n\rr$ is the identity.

By now we have a function $T$ from the finitely generated group $\Aut_G(G*F_n)$ onto the subgroup $\Stab(\phi_0)$, which restricts to the identity on $\Stab(\phi_0)$. Even though $T$ is not a homomorphism, our next goal is to prove that it is sufficiently close to being a homomorphism that we are still able to deduce finite generation for $\Stab(\phi_0)$.

Consider the generating set $X$ for $\Aut_G(G*F_n)$ from Proposition~\ref{prop:gens}, which our hypotheses ensure really is a generating set. Note that the subset $Y\subset X$ is contained in $\Stab(\phi_0)$. Our goal is to prove that $\langle Y\rangle=\Stab(\phi_0)$. This will finish the proof, since $\langle Y\rangle$ is finitely generated thanks to $G$ being finitely generated. It is sufficient to show that for each generator $\xi\in X$ and each $\alpha\in \Aut_G(G*F_n)$ we have $T(\alpha\xi)=T(\alpha)\psi$ for some $\psi\in\langle Y\rangle$. To see this is sufficient, note that for $\alpha \in \Stab(\phi_0)$ with $\alpha = \xi_1 \cdots \xi_p$ ($\xi_i\in X$), there exist $\psi_1, \ldots, \psi_p \in \langle Y\rangle$ such that
\[\alpha = T(\alpha) = T(\xi_1 \cdots \xi_p) = T(\xi_1 \cdots \xi_{p-1})\psi_p = \cdots = T(\xi_1) \psi_2 \cdots \psi_p = \psi_1 \cdots \psi_p \in \langle Y \rangle.\]
We now work through the five types of generators.

\medskip

\noindent Type (i): Let $\xi$ be given by the basis permutation $x_i\mapsto x_{\sigma(i)}$ for $\sigma\in S_n$. Then
\[
T(\alpha\xi)(x_i) = t(\alpha\xi(x_i)) = t(\alpha(x_{\sigma(i)})) = T(\alpha)(x_{\sigma(i)}) = T(\alpha)\xi(x_i)
\]
for all $i$, so $T(\alpha\xi)=T(\alpha)\xi$.

\medskip

\noindent Type (ii): Let $\xi$ fix $x_2,\dots,x_n$ and invert $x_1$. Then
\[
T(\alpha\xi)(x_1) = t(\alpha\xi(x_1)) = t(\alpha(x_1^{-1})) = t(\alpha(x_1)^{-1}).
\]
Say $\alpha(x_1)=(f,g)$, and let $\psi\in Y$ be the element sending $x_1$ to $g^{-1}x_1 g$ and fixing $x_2,\dots,x_n$. Then, using \eqref{eq:t:inverse}:
\begin{align*}
    t(\alpha(x_1)^{-1}) &= g^{-1} (t(\alpha(x_1)))^{-1} g = g^{-1} (T(\alpha)(x_1))^{-1} g = g^{-1} T(\alpha)(x_1^{-1}) g\\
    & = T(\alpha)(g^{-1} x_1^{-1} g) = T(\alpha) \xi (g^{-1} x_1 g) = T(\alpha)\xi\psi(x_1).
\end{align*}
For any $i\ne 1$ we have $T(\alpha\xi)(x_i) = T(\alpha)\xi\psi(x_i)$ trivially, so $T(\alpha\xi)=T(\alpha)\xi\psi$.

\medskip

\noindent Type (iii): Let $\xi$ fix $x_2,\dots,x_n$ and send $x_1$ to $g^{-1}x_1 g$ ($g\in G$). Then
\begin{align*}
T(\alpha\xi)(x_1) &= t(\alpha\xi(x_1)) = t(\alpha(g^{-1}x_1 g)) = t(g^{-1}\alpha(x_1) g) \\
&= g^{-1} t(\alpha(x_1)) g = g^{-1}(T(\alpha)(x_1)) g = T(\alpha)(g^{-1} x_1 g) = T(\alpha)\xi(x_1),
\end{align*}
and for any $i\ne 1$ we have $T(\alpha\xi)(x_i) = T(\alpha)\xi(x_i)$ trivially, so $T(\alpha\xi)=T(\alpha)\xi$.

\medskip

\noindent Type (iv): Let $\xi$ fix $x_2,\dots,x_n$ and send $x_1$ to $x_1 x_2$. Then $T(\alpha\xi)(x_1) = t(\alpha\xi(x_1)) = t(\alpha(x_1 x_2)) = t(\alpha(x_1)\alpha(x_2))$. Write $h_1=\alpha(x_1)$ and $h_2=\alpha(x_2)$, and for $i=1,2$ write $h_i=(f_i,g_i)$ in the external semidirect product notation, so $h_1 h_2=(f_1(g_1 f_2 g_1^{-1}),g_1 g_2)$. We get $t(\alpha(x_1)\alpha(x_2))=t(h_1 h_2)=f_1(g_1 f_2 g_1^{-1})$. Let $\psi$ be the element sending $x_1$ to $x_1(g_1 x_2 g_1^{-1})$ and fixing $x_2,\dots,x_n$. This is the result of conjugating the type (iv) generator $\xi$ by the type (iii) generator sending $x_2$ to $g_1 x_2 g_1^{-1}$ and fixing $x_1, x_3, \ldots, x_n$, hence $\psi \in \langle Y \rangle$. Note that $T(\alpha)(\psi(x_1)) = T(\alpha)(x_1(g_1 x_2 g_1^{-1})) = T(\alpha)(x_1)T(\alpha)(g_1 x_2 g_1^{-1}) = f_1 t(g_1 h_2 g_1^{-1}) = f_1 g_1 f_2 g_1^{-1}$. Thus, $T(\alpha\xi)=T(\alpha)\psi$ on $x_1$, and trivially also on the other $x_i$, so these are equal as desired. An analogous argument shows $T(\alpha\xi^{-1})=T(\alpha)\psi^{-1}$.

\medskip

\noindent Type (v): Let $\xi$ fix $x_2,\dots,x_n$ and send $x_1$ to $x_1 g$ ($g\in G$). Then
\[
T(\alpha\xi)(x_1)=t(\alpha\xi(x_1))=t(\alpha(x_1 g))=t(\alpha(x_1)g)=t(\alpha(x_1))=T(\alpha)(x_1),
\]
and trivially $T(\alpha\xi)(x_i)=T(\alpha)(x_i)$ for all $i\ne 1$, so $T(\alpha \xi) = T(\alpha)$ and we are done.
\end{proof}


\section{Proof of the main results}\label{sec:main_proof}

Now we can use the results from the previous sections to prove Theorem~\ref{thrm:autg_pbhc}, after which most of the other main results will follow quickly. We will save the proofs of some of the main corollaries for the next two sections.

\begin{proof}[Proof of Theorem~\ref{thrm:autg_pbhc}]
Let $G$ be a finitely presented simple MIF group and let $n\ge 2$. Consider $\Aut_G(G*F_n)$ acting on $\Hom_G(G*F_n,G)$, and we claim this action is of type~(A). It is faithful by Lemma~\ref{lem:faithful}. The group $\Aut_G(G*F_n)$ is finitely presented by Proposition~\ref{prop:fp}. The stabilizer of $\phi_0\in \Hom_G(G*F_n,G)$ is finitely generated by Proposition~\ref{prop:fg_stab}, and the action is transitive by Lemma~\ref{lem:trans} so every point stabilizer is finitely generated. Finally, there are finitely many orbits of pairs by Proposition~\ref{prop:2trans}. We conclude that $\Aut_G(G*F_n)$ admits a type~(A) action, and hence satisfies the permutational Boone--Higman conjecture as desired. It therefore also satisfies the Boone--Higman conjecture by \cite{zaremsky_fp_tbt}.
\end{proof}

\begin{proof}[Proof of Theorem~\ref{thrm:autfn}]
For $n=1$ we have $\Aut(F_1)\cong \Z/2\Z$ so there is nothing to prove. Now assume $n \ge 2$. Fix some finitely presented simple MIF group, for instance Thompson's group $V$. Note that $\Aut(F_n)$ embeds in $\Aut_V(V*F_n)$, which admits an action of type~(A) by Theorem~\ref{thrm:autg_pbhc}. Thus, $\Aut(F_n)$ satisfies PBH for all $n$.
\end{proof}

\begin{proof}[Proof of Theorem~\ref{thrm:mif_pbhc}]
First note that (i) implies (ii) by \cite[Theorem~A]{zaremsky_fp_tbt}. By \cite[Proposition~5.4]{bbmz_hyp}, any twisted Brin--Thompson group admits a faithful highly transitive action, so (ii) implies (iii). Any finitely presented simple group admitting a highly transitive action is MIF thanks to \cite[Theorem~5.9]{hull16} (see also \cite[Proposition~A.1]{leboudec22}), so (iii) implies (iv). It remains to prove that (iv) implies (i). Let $G$ be a finitely presented simple MIF group. Since $G$ embeds in $\Aut_G(G*F_2)$, for example by sending $g$ to the $G$-automorphism that takes $x_1$ to $x_1 g$ and fixes $x_2$, and $\Aut_G(G*F_2)$ admits an action of type~(A) by Theorem~\ref{thrm:autg_pbhc}, we conclude that $G$ satisfies PBH.
\end{proof}

The main outstanding general question is the following.

\begin{question}\label{quest:non_mif}
For $G$ a finitely presented simple group and $n\ge 2$, is the quotient of $\Aut_G(G*F_n)$ by the kernel of its action on $\Hom_G(G*F_n,G)$ finitely presented?
\end{question}

Just to recall the context, $\Aut_G(G*F_n)$ itself is finitely presented, and we showed that if $G$ is MIF then this kernel is trivial, so the result holds in this case. By Observation~\ref{obs:nonmif}, if $n=1$ then the kernel is trivial regardless, but for $G$ non-MIF (even for the easiest example we are aware of: Thompson's group $T$ \cite[Proposition 4.7]{leboudec22}) and $n\ge 2$ it is definitely non-trivial, so it is unclear what to expect.

If the answer to Question~\ref{quest:non_mif} is yes, then BH is equivalent to PBH, and finitely presented twisted Brin--Thompson groups are universal among all finitely presented simple groups. To see why this follows note that, on the one hand, if a group admits an action that satisfies all the aspects of being of type~(A) except faithfulness, then modding out the kernel of the action yields a faithful action that satisfies all the aspects of being of type~(A) except possibly finite presentability. Also, the copy of $G$ inside $\Aut_G(G*F_n)$ coming from translating $x_1$ by elements of $G$ intersects the kernel of the action on $\Hom_G(G*F_n,G)$ trivially, so $G$ still embeds in the quotient.

A related question, that seems to be of interest in its own right, is the following.

\begin{question}\label{quest:jacobson}
For $G$ a finitely presented simple group, is the quotient of $G*F_n$ by the Jacobson radical $J_G(G*F_n)$ finitely presented? That is, is $J_G(G*F_n)$ finitely normally generated in $G*F_n$?
\end{question}

While it is not immediately clear to us whether Questions~\ref{quest:non_mif} and~\ref{quest:jacobson} are equivalent for $n\ge 2$, the latter feels like the more natural first step. Indeed, even for Thompson's group $T$ this is unclear to us.

Yet another relevant question, which is more succinct than Question~\ref{quest:non_mif} but probably more difficult to prove directly, is the following.

\begin{question}\label{quest:embed_in_mif}
Does every finitely presented simple group embed in a finitely presented simple MIF group?
\end{question}

If the answer is yes, then BH is equivalent to PBH, and finitely presented twisted Brin--Thompson groups are universal among finitely presented simple groups. If the answer is no, then PBH is false. Thus either answer would be very interesting.

\begin{remark}\label{rmk:mif}
Now that we know that all finitely presented simple groups that are highly transitive, or more generally MIF, satisfy PBH, it is worth recording which finitely presented simple groups are known to have these properties. Note that for finitely generated simple groups, highly transitive implies MIF \cite[Theorem~5.9]{hull16} \cite[Proposition~A.1]{leboudec22}, and in fact we are not aware of any finitely presented simple groups that are MIF but not highly transitive (there exist finitely generated examples \cite[Theorem~4.13]{confined:hightrans}).

First there are the ones that are ``obviously'' highly transitive, hence MIF. This includes any finitely presented simple groups of homeomorphisms of the Cantor space $\{1,\dots,n\}^\N$ that contain the commutator subgroup of the Higman--Thompson group $V_n$. In particular this includes Thompson's group $V$ itself, along with all finitely presented commutator subgroups of R\"over--Nekrashevych groups $V_n(G)$ \cite[Theorem~4.7]{nekrashevych18}. Moreover, all twisted Brin--Thompson groups themselves are highly transitive \cite[Proposition~5.4]{bbmz_hyp}, hence MIF.

The finitely presented simple Burger--Mozes groups $\Burger$ from \cite{burgermozes} are not ``obviously'' highly transitive and MIF, but it turns out that they do have these properties; we thank Adrien Le Boudec for this argument. By construction, $\Burger$ is dense in a subgroup of the automorphism group of a locally finite tree $T$ that is $2$-transitive on the boundary \cite[Remark~5]{burgermozes}. It follows from density that the action of $\Burger$ on $\partial T$ is extremely proximal. By \cite[Theorem~B]{mif:trees} it suffices to show that the action of $\Burger$ on $\partial T$ is topologically free. Suppose that this is not the case; this means that there exists a proper open set $U \subset \partial T$ such that the subgroup $\Burger_U$ of elements of $\Burger$ supported on $U$ is non-trivial. By extreme proximality, there exist $g_1, g_2, g_3 \in \Burger$ such that the $g_i U$ are pairwise disjoint. Then the $g_i$-conjugates of $\Burger_U$ generate a direct product. However, $\Burger$ has cohomological dimension $2$, so it cannot contain a direct product of three non-trivial groups. This shows that Burger--Mozes groups are highly transitive, hence MIF, and thus satisfy PBH. In fact, this latter fact was recently established, in a completely different way, in \cite{bux_boonehigman}.
\end{remark}

To the best of our knowledge, when this paper was first written the only existing source of finitely presented infinite simple groups for which PBH remained open was non-affine Kac--Moody groups over finite fields \cite{caprace09}. These are closely related to Burger--Mozes groups in some aspects, but with key differences. One of these differences is the behavior of the action on the boundary, which in turn is responsible for the difference in their second bounded cohomology \cite[Theorem~1.8]{capracefujiwara}. In the above proof that Burger--Mozes groups are highly transitive, this was the main input, so a different argument would be needed.

\begin{question}
\label{q:km}
Are finitely presented simple Kac--Moody groups highly transitive? Are they MIF? Do they satisfy PBH?
\end{question}

Ekaterina Rybak \cite{rybak} recently answered Question~\ref{q:km} by showing that such groups are MIF, and hence they satisfy PBH (although it remains open if they are highly transitive).

\begin{remark}
     Since the first version of this paper, two new examples of finitely presented simple groups have been constructed. The first construction, by Amir--Lazarovich, consists of lattices in products of a tree and a Davis complex \cite{amir:lazar}. Projecting onto the tree factor, one can run the same argument as the one above for Burger--Mozes groups, to prove that they are MIF (this was already noticed in \cite{rybak}). The second construction is in the realm of lattices in exotic two-dimensional irreducible Euclidean buildings. For type $\widetilde{A}_2$, these are conjectured to be all virtually simple \cite{conj:a2}, although no example has appeared to date. For type $\widetilde{C}_2$, Titz Mite--Witzel \cite{titzmite:witzel} constructed non-residually finite examples, and the subsequent normal subgroup theorem of L{\'e}cureux--Witzel \cite{lecureux:witzel} implies that they are virtually simple. They were also recently shown to be MIF, by Le Bars--Leibtag--Vigdorovich \cite{lebars}.
\end{remark}

\begin{remark}
\label{rem:finitenessproperties}
We proved that if $G$ is finitely presented, simple, and MIF, then the twisted Brin--Thompson group $SV_\Gamma$ is also finitely presented, where $\Gamma = \Aut_G(G*F_n)$ and $S = \Hom_G(G*F_n,G)$. It is natural to conjecture that if $G$ has type $\F_m$ then so does $SV_\Gamma$. Here recall that a group has \newword{type $\boldsymbol{\F}_m$} if it has a classifying space with finite $m$-skeleton, so type $\F_2$ means finitely presented.

However this is not the case in general. It is shown in \cite[Corollary~1.6]{fournier2025finiteness} that when $G$ is a finitely presented simple Burger--Mozes group, which is of type $\F_\infty$ and even of type $\F$, the corresponding $SV_\Gamma$ is \emph{not} of type $\F_\infty$.
\end{remark}


\section{Permanence properties}\label{sec:permanence}

In this section we use our construction to prove that the PBH conjecture is preserved under two important group-theoretic constructions, namely free products and regular wreath products. We start with free products; the main idea of the proof is due to Xiaolei Wu.

\begin{proof}[Proof of Corollary~\ref{cor:freeprod_pbhc}]
Let $A$ and $B$ be groups that satisfy PBH. By \cite[Proposition~5.5]{zaremsky_fp_tbt}, the direct product $A\times B$ also satisfies PBH, so $A\times B$ embeds in a finitely presented twisted Brin--Thompson group, which we will just denote by $G$. Since $G$ is finitely presented, simple, and MIF, Theorem~\ref{thrm:autg_pbhc} says that $\Aut_G(G*F_n)$ satisfies PBH. Now observe that (for $n\ge 2$) $\Aut_G(G*F_n)$ contains $G*G$, and hence $A*B$, for example as the subgroup of $G$-automorphisms generated by the copy of $G$ fixing $x_2,\dots,x_n$ and sending $x_1$ to $x_1 g$ for $g\in G$, and the copy of $G$ fixing $x_2,\dots,x_n$ and sending $x_1$ to $x_1(x_2^{-1}gx_2)$ for $g\in G$.
\end{proof}

Given that satisfying PBH is preserved under direct and free products, the following is the natural next step:

\begin{question}\label{question:graph}
Let $G$ be a graph product over a finite graph such that each vertex group satisfies PBH. Does $G$ satisfy PBH?
\end{question}

Another natural question one might have is whether satisfying PBH is preserved under amalgamated products, or more general graph of groups constructions. When the corresponding Bass--Serre tree is locally finite, some results are contained in \cite{bux_boonehigman}, but in the general case this is not true. Indeed, solvability of the word problem for an amalgamated product is tied to solvability of the membership problem for the amalgamated subgroup, and so it is easy to construct examples of groups satisfying PBH but whose amalgamated product does not even have solvable word problem.

\medskip

Now we move to wreath products. We first prove a general criterion for permutational wreath products, and then use Theorem~\ref{thrm:autg_pbhc} to apply it to regular wreath products.

\begin{proposition}
\label{prop:wreath:typeA}
    Let $A$ be a group satisfying the permutational Boone--Higman conjecture. Let $B \curvearrowright X$ be an action of type (A). Then the permutational wreath product $A \wr_X B$ satisfies the (permutational) Boone--Higman conjecture.
\end{proposition}

\begin{proof}
    Up to embedding, we may assume that $A$ itself admits an action of type (A), say $A \curvearrowright Y$. We consider the natural imprimitive action $A \wr_X B \curvearrowright Y \times X$, where we see $Y \times X$ as a disjoint union of copies of $Y$ indexed by $X$, with each summand of $\oplus_X A$ acting on the corresponding copy of $Y$, and $B$ shifting the copies according to its action on $X$. We claim that this action is of type (A).

    First, $A \wr_X B$ is finitely presented by \cite{cornulier06}. Now let $(y, x) \in Y \times X$; an element $((a_x)_{x \in X}, b)$ fixes it if and only if $b$ fixes $x$ and $a_x$ fixes $y$. Now let $x_1, \ldots, x_n$ be a set of orbit representatives for the action of $\Stab_B(x)$ on $X \setminus \{ x \}$; note that this is finite because $B \curvearrowright X$ is of type (A). Then $\Stab_{A \wr_X B}(y,x)$ is generated by $\Stab_B(x)$ together with the subgroup of $\oplus_X A$ with $\Stab_A(y)$ at the $x$ coordinate, $A$ at the $x_i$ coordinate for $i = 1, \ldots, n$, and the trivial group at all other coordinates. Since both actions are of type (A), these are both finitely generated groups, so $\Stab_{A \wr_X B}(y, x)$ is finitely generated.

    Finally, we show that the action has finitely many orbits of pairs. Using the action of $B$, any pair $(y_1, x_1), (y_2, x_2)$ can be moved so that $\{ x_1, x_2 \}$ is mapped to a representative of the finitely many orbits of pairs for the action of $B$ on $X$. Now if $x_1 \neq x_2$, we can use the action of the corresponding copies of $A$ to move $y_1, y_2$ to representatives of the finitely many orbits for the action of $A$ on $Y$. If $x_1 = x_2$, then we can again use the action of this copy of $A$ to move the pair $\{ y_1, y_2 \}$ to representatives of the finitely many orbits of pairs for the action of $A$ on $Y$.
\end{proof}

\begin{proof}[Proof of Corollary \ref{cor:wreath_pbhc}]
    Let $A, B$ be two groups satisfying PBH. Then $B$ embeds in a finitely presented twisted Brin--Thompson group, which we denote by $G$, and it suffices to show that $A \wr G$ satisfies PBH. Since $G$ is finitely presented, simple and MIF, Theorem~\ref{thrm:autg_pbhc} says that the action of $\Gamma \coloneqq \Aut_G(G * F_n)$ on $S \coloneqq \Hom_G(G * F_n, G)$ is of type (A). So Proposition~\ref{prop:wreath:typeA} implies that $A \wr_S \Gamma$ satisfies PBH.
    
    We claim that $A \wr G$ embeds into $A \wr_S \Gamma$. To see this, it suffices to show that $G$ embeds into $\Gamma$ acting freely on some orbit in $S$. We consider the copy of $G$ in $\Gamma$ given by the generators of type (v) (see Section~\ref{sec:group}), that is $x_1 \mapsto x_1 g$. This acts freely on the orbit in $S$ of the basepoint $\phi_0$, indeed for all $g \in G \setminus \{ 1 \}$ we have $\phi_0(x_1 g) = g \neq 1 = \phi_0(x_1)$, hence $g$ does not fix $\phi_0$. As a remark, this orbit is naturally a copy of $G$ in $S=\Hom_G(G*F_n,G)$ viewed as a group, and the action of $G$ is by translation, so the embedding of $A\wr G$ is particularly natural.
\end{proof}

\begin{proof}[Proof of Corollary \ref{cor:freesolvable}]
    We use $S_{r, d}$ to denote the free $d$-solvable group of rank $r$, and proceed by induction on $d$. For $d = 1$, this is just the free abelian group $\Z^r$, which satisfies PBH, for instance because it embeds into Thompson's group $V$. Now suppose that $S_{r, d}$ satisfies PBH; then by Corollary \ref{cor:wreath_pbhc} also $\Z^r \wr S_{r, d}$ satisfies PBH, which by the Magnus embedding theorem \cite{magnus} contains a copy of $S_{r, d+1}$.
\end{proof}

\section{Boone--Higman for specific groups}\label{sec:AutFn}

In this section we prove Corollary~\ref{cor:autfn_pbhc}, showing that many specific groups of interest satisfy the (P)BH conjecture.

\begin{proof}[Proof of Corollary~\ref{cor:autfn_pbhc}]
By Theorem~\ref{thrm:autfn}, we know that $\Aut(F_n)$ satisfies PBH for all $n$. This immediately implies that ((extended) loop) braid groups satisfy PBH, since we have embeddings $\mathcal{B}_n \to \mathcal{LB}_n \to \mathcal{LB}_n^{ext} \to \Aut(F_n)$; see, e.g., \cite[Section~4]{loopbraid:survey}. Moreover, the pure ribbon braid group and the pure loop braid group are isomorphic \cite[Proposition~5.10]{loopbraid:survey}, which implies that $\mathcal{RB}_n$ is commensurable to a group satisfying PBH and thus satisfies PBH \cite[Proposition~5.6]{zaremsky_fp_tbt}.

For Artin groups, those of type $B_n=C_n$ are the ``annular braid groups'', i.e., the mapping class groups of punctured annuli, and the Artin groups of type~$\widetilde{A}_n$ embed into these~\cite{kent02}.  The Artin groups of type~$D_n$ embed into the mapping class groups of certain finite type surfaces with nonempty boundary~\cite[Theorem~1]{perron92}. The Artin groups of type $I_2(m)$ are all commensurable to the braid group on three strands~\cite[Theorem~5]{cumplido22} and thus satisfy PBH~\cite[Proposition~5.6]{zaremsky_fp_tbt}. The Artin groups of type~$\widetilde{C}_n$ embed into braid groups \cite{allcock, charney:crisp}. Finally, the Artin group of type $H_3$ embeds into the Artin group of type $D_6$ \cite{crisp}.

We are left to prove the results on mapping class groups. Surfaces $\Sigma$ are always assumed to be of finite type. If $\Sigma$ is non-orientable, then its mapping class group embeds into the mapping class group of its orientable double cover $\widetilde{\Sigma}$ \cite[Proposition 26]{paris:notes}. Moreover, if $\Sigma$ has punctures or boundary, so does $\widetilde{\Sigma}$, and if $\Sigma$ has demigenus at most three, then $\widetilde{\Sigma}$ has genus at most two. Hence from now on we may assume that all surfaces are \emph{orientable}. We will use the fact that satisfying PBH is invariant under commensurability \cite[Proposition~5.6]{zaremsky_fp_tbt}. In particular, it suffices to consider mapping class groups, and the results for extended mapping class groups will follow.

First we will prove the result for surfaces with at least two punctures. Let $\Sigma$ be a closed surface with at least one puncture, and let $*$ be a marked point on $\Sigma$, so $\MCG(\Sigma, *)$ is the group of all self-homeomorphisms of $\Sigma$ that preserve the basepoint $*$, up to pointed isotopy. Then $\MCG(\Sigma, *)$ embeds into $\Aut(\pi_1(\Sigma))$ by the Dehn--Nielsen--Baer Theorem (see \cite[Theorem~8.8]{farbmargalit} and the paragraph following it), which is just $\Aut(F_n)$ for some $n$, and therefore satisfies PBH. Moreover $\MCG(\Sigma, *)$ is the finite index subgroup of $\MCG(\Sigma \setminus \{ * \})$ consisting of mapping classes that fix the puncture $*$, so $\MCG(\Sigma \setminus \{ * \})$ satisfies PBH. This shows that $\MCG(\Sigma)$ satisfies PBH whenever $\Sigma$ is a closed surface with at least two punctures.

For the case of one puncture, first let $\Sigma$ be a closed surface, which we may assume to be of genus $g \geq 1$, and let $\Sigma_* \coloneqq \Sigma \setminus \{ * \}$, for some $* \in \Sigma$. Then $\Sigma_*$ has a double cover $\widetilde{\Sigma}$, which must have genus $2g-1$ and two punctures. Let $L \leq \MCG(\Sigma_*)$ denote the subgroup of mapping classes that can be lifted to $\widetilde{\Sigma}$. Identifying $\MCG(\Sigma_*)$ with $\Aut(\pi_1(\Sigma))$ via the Dehn--Nielsen--Baer Theorem \cite[Theorem 8.8]{farbmargalit}, $L$ corresponds to the subgroup of automorphisms that preserve $H \leq \pi_1(\Sigma)$ where $H$ is the index-$2$ subgroup defined by the double cover. It follows that $L$ has finite index in $\MCG(\Sigma_*)$. Finally, we note that each element in $L$ has a unique lift that fixes the two punctures in~$\widetilde{\Sigma}$. This defines a homomorphism $L \to \MCG(\widetilde{\Sigma})$, which is injective by a version of the Birman--Hilden Theorem \cite[Corollary~4]{aramayona09}. Since $\MCG(\widetilde{\Sigma})$ satisfies PBH by the previous paragraph, this proves PBH for $L$. Since $L$ has finite index in $\MCG(\Sigma_*)$, this proves PBH for $\MCG(\Sigma_*)$.

Next let $\Sigma$ be a surface with non-empty boundary. Let $\Sigma'$ be the surface obtained from $\Sigma$ by attaching a twice-punctured disk to every boundary component of $\Sigma$. Then $\MCG(\Sigma)$ embeds into $\MCG(\Sigma')$ \cite[Theorem~3.18]{farbmargalit}, which is handled by the previous paragraph.

The only remaining case to prove is when $\Sigma$ is the closed surface of genus $2$. Let $h \in \MCG(\Sigma)$ be the hyperelliptic involution. The subgroup $\langle h \rangle$ is central and of order $2$, and $\MCG(\Sigma)/\langle h \rangle \cong \MCG(\Sigma')$, where $\Sigma'$ is a sphere with six punctures \cite{birmanhilden}. Moreover, $\MCG(\Sigma)$ is residually finite \cite[Theorem~6.11]{farbmargalit}, and so there exists a finite-index subgroup $H < \MCG(\Sigma)$ intersecting $\langle h \rangle$ trivially, hence mapping injectively into $\MCG(\Sigma')$. This shows that $\MCG(\Sigma)$ is commensurable to $\MCG(\Sigma')$, which satisfies PBH.
\end{proof}

A natural question is whether a similar approach could prove the (P)BH conjecture for $\Out(F_n)$, and for mapping class groups of closed surfaces. As indicated in Remark~\ref{rmk:garion}, it seems very difficult to find an action of type~(A) for $\Out(F_n)$ itself, so we would want to embed it in some analog of $\Aut_G(G*F_n)$. However, inner automorphisms of $F_n$ do not interact nicely with $G$-automorphisms, so it is not clear whether this is possible. If a version for $\Out(F_n)$ were possible, it is likely that a similar argument would also cover the mapping class group of a closed surface $\Sigma$, since this is an index-$2$ subgroup of $\Out(\pi_1(\Sigma))$.

\begin{remark}\label{rem:othergroups}
Knowing that every $\Aut(F_n)$ satisfies (P)BH not only establishes the (P)BH conjecture for the above groups, for which it was not previously known, but also recovers the (P)BH conjecture for some important classes of groups for which it was previously known.

First, we recover the fact that virtually compact special groups satisfy PBH. Indeed, by commensurability invariance it suffices to show this for RAAGs \cite{special}, and RAAGs embed into $\Aut(F_n)$ \cite{humphries} (in fact they even embed into braid groups \cite{raagsinbraids}), so we conclude from Theorem~\ref{thrm:autfn} that they satisfy PBH.
It was already known that RAAGs satisfy PBH, in fact for various reasons. First, they embed into $\mathrm{GL}_n(\mathbb{Z})$ \cite{raags:linear}, which embeds in a highly transitive finitely presented simple group thanks to work of Scott \cite{scott84}, and hence satisfies PBH by Theorem~\ref{thrm:mif_pbhc}. More directly, every RAAG embeds in some $nV$ \cite{belk20}, in fact into $2V$ \cite{salo}, and thus satisfies PBH.

Next, we recover the fact that (finitely generated free)-by-cyclic groups satisfy PBH, which was recently established, in a completely different way, in \cite{bux_boonehigman}. In fact we will prove something more general in Proposition~\ref{prop:extensions} below. Thanks to \cite{linton}, we deduce that, more generally, finitely generated free-by-cyclic groups satisfy the PBH conjecture, and hence also finitely generated virtually free-by-cyclic groups. This applies to many groups of geometric interest, including large families of one-relator groups \cite{kielak:linton, fisher}.
\end{remark}

\begin{proposition}
\label{prop:extensions}
    Let $A$ be a centerless group such that $\Aut(A)$ satisfies PBH, (for instance, a finitely generated free group of rank at least $2$, or the fundamental group of a closed surface of genus at least $2$). Let $G$ be a group containing $A$ as a normal subgroup such that $G/A$ satisfies PBH. Then $G$ satisfies PBH.
\end{proposition}

\begin{proof}
    The action of $G$ on $A$ by conjugation induces a homomorphism $G \to \Aut(A)$. Consider the product $G \to \Aut(A) \times G/A$ with the quotient map. The kernel of the right component is precisely $A$, but since $A$ is centerless, the restriction of the left component to $A$ is injective. Hence the product $G \to \Aut(A) \times G/A$ is injective. Now by assumption $\Aut(A)$ and $G/A$ satisfy PBH, hence so does their product by \cite[Proposition~5.5]{zaremsky_fp_tbt}.
\end{proof}

Finally, let us discuss finitely generated $3$-manifold groups.

\begin{question}
    Let $M$ be a $3$-manifold with finitely generated fundamental group. Does $\pi_1(M)$ satisfy the Boone--Higman conjecture?
\end{question}

The results in this paper allow us to reduce to the case of non-Seifert fibered closed graph manifolds.

\begin{corollary}
\label{cor:graph:mfld}
    Suppose that all fundamental groups of non-Seifert fibered closed graph manifolds satisfy PBH. Then all finitely generated $3$-manifold groups satisfy PBH.
\end{corollary}

\begin{proof}
Let $M$ be a (connected) $3$-manifold such that $\pi_1(M)$ is finitely generated. By \cite[Proposition 2.2]{kielak:linton:mfld}, there exist compact orientable aspherical $3$-manifolds $M_1, \ldots, M_n$ and a free group $F$ such that a finite-index subgroup of $\pi_1(M)$ is isomorphic to the free product of $F$ and the $\pi_1(M_i)$. By commensurability \cite[Proposition~5.6]{zaremsky_fp_tbt} and Corollary~\ref{cor:freeprod_pbhc}, we may now assume that $M$ itself is compact orientable and aspherical.

If $M$ has non-empty boundary, then it is homotopy equivalent to its spine (see e.g., \cite[page~560]{casler}), which is an aspherical $2$-complex, hence $\pi_1(M)$ has rational cohomological dimension at most $2$. Now \cite[Theorem~1.1]{kielak:linton:mfld} applies and gives that $\pi_1(M)$ is virtually free-by-cyclic, hence it satisfies PBH by Remark~\ref{rem:othergroups} (or by \cite{bux_boonehigman} and \cite{linton}).

If $M$ is closed, and not a graph manifold, then by one of the virtual fibering theorems of Agol, Wise and Przytycki--Wise \cite{agol1, agol2, mixed, wise}, $M$ virtually fibers. Again by commensurability \cite[Proposition~5.6]{zaremsky_fp_tbt} we may assume $M$ is fibered, so there exists a closed surface $\Sigma$ such that $\pi_1(M) \cong \pi_1(\Sigma) \rtimes_\alpha \Z$, for $\alpha \in \Aut(\pi_1(\Sigma))$. If $\Sigma$ has genus at least $2$, then $\pi_1(M)$ satisfies PBH by Proposition~\ref{prop:extensions}. If $\Sigma$ is the torus, then $\pi_1(M)$ embeds into $\mathrm{GL}_n(\Z)$ for some $n$, hence satisfies PBH.

It remains to consider closed Seifert fibered manifolds. For such a manifold $M$, we may assume that $\pi_1(M)$ is infinite, in which case up to finite index it is a bundle $S^1 \to M \to \Sigma$, where $\Sigma$ is a closed surface with $\chi(\Sigma) \leq 0$. Now by \cite[Proof of Lemma 4.17]{vresp}, there exists a finite-index subgroup $G_e < \pi_1(M)$ that embeds into $H \times \pi_1(\Sigma)$, where $H$ is a finitely generated nilpotent group. Hence $H$ embeds into some $\mathrm{GL}_n(\Z)$ (in fact, the $H$ appearing in the proof is just the integral Heisenberg group $\mathrm{H}_3(\Z)$) and so it satisfies PBH, and $\pi_1(\Sigma)$ satisfies PBH as we have seen above. By \cite[Proposition 5.5]{zaremsky_fp_tbt}, the direct product $H \times \pi_1(\Sigma)$ satisfies PBH, hence so does $G_e$, hence $\pi_1(M)$ by commensurability \cite[Proposition 5.6]{zaremsky_fp_tbt}.
\end{proof}

\begin{remark}
Note that the proof above also covers closed graph manifolds that virtually fiber over the circle (which includes all non-positively curved graph manifolds \cite{svetlov}), so the problem is actually reduced to proving PBH for fundamental groups of non-fibered closed graph manifolds. Up to finite index, these are fundamental groups of graphs of groups with $F_n \times \Z$ vertex groups and $\Z^2$ edge groups \cite[Lemma~2.1]{aspects}, and we already mentioned that the question of whether graphs of groups with vertex groups satisfying PBH satisfy PBH is a delicate one. On the other hand, these groups have algebraic features that could make them amenable to arguments using closure properties for PBH (indeed a recent example can be found in \cite{aspects}); the main issue is that \emph{infinitely generated} free groups often play a role in these descriptions.
\end{remark}

\bibliographystyle{alpha}

\end{document}